\documentclass[11pt,a4paper,reqno]{amsart}
\usepackage{amsmath}
\usepackage{amssymb}
\usepackage{amsthm}
\usepackage{amsxtra}
\usepackage{amscd, amsrefs}
\usepackage{tikz}

\usepackage{microtype}

 \usepackage{hyperref}
\DeclareMathOperator{\dist}{dist}
\DeclareMathOperator{\supp}{supp}

\theoremstyle{plain}
\newtheorem{thm}{Theorem}[section]
\newtheorem{lem}[thm]{Lemma}
\newtheorem{prop}[thm]{Proposition}
\newtheorem{cor}[thm]{Corollary}

\theoremstyle{definition}

\theoremstyle{remark}
\newtheorem{remark}[thm]{Remark}

\numberwithin{equation}{section}

\begin{document}
\date{\today}

\title
[Square function and local smoothing estimates]{A trilinear approach to square function and local smoothing estimates for the wave operator}

\author[J. Lee]{Jungjin Lee}

\address{Department of Mathematical Sciences, School of Natural Science, Ulsan National Institute of Science and Technology, UNIST-gil 50, Ulsan 44919, Republic of Korea}
\email{jungjinlee@unist.ac.kr}



\thanks{
The author was supported in part by NRF grant No. 2017R1D1A1B03036053 (Republic of Korea).}

\begin{abstract}
The purpose of this paper is to improve the known estimates for Mockenhaupt's square function in $\mathbb R^3$ and for Sogge's local smoothing in $\mathbb R^{2+1}$ spacetime. For this we use the trilinear approach of S. Lee and A. Vargas for the cone multiplier with some trilinear estimates obtained from the $\ell^2$ decoupling theorem and multilinear restriction theorem.
\end{abstract}

\maketitle

%

\section{Introduction}

Let $\Gamma = \{ (\xi,\tau) \in \mathbb R^2 \times \mathbb R : \tau = |\xi|,~ 1 \le \tau \le 2 \}$ be a truncated light cone in $\mathbb R^3$. For given small $0< \delta <1 $, let $\Gamma_\delta$ denote the $\delta$-neighborhood of $\Gamma$. Let $f$ be a function on $\mathbb R^3$ whose Fourier transform is supported in $\Gamma_\delta$.
We partition $\Gamma_\delta$ into $O(\delta^{-1/2})$ sectors
\(
\Theta =  \{ (\xi, \tau) \in \Gamma_\delta : \xi/|\xi| \in \theta \}
\)
corresponding to an arc $\theta$  of angular length $O(\delta^{1/2})$ in the unit circle, and  
let $\mathbf \Pi_\delta$ denote the collection of such sectors. We take a collection of Schwartz functions \(  \Xi_{\Theta} \) so that its Fourier transform $\widehat \Xi_\Theta$ is supported on a neighborhood of $\Theta$ and  $\{\widehat \Xi_\Theta \}_{\Theta \in \mathbf \Pi_\delta}$ forms a partition of unity of $\Gamma_\delta$. 
The square function $S_\delta f$ is defined as
\[
S_\delta f =  \Big(\sum_{\Theta \in \mathbf \Pi_\delta}  |f_\Theta|^2 \Big)^{1/2}
\]
where $f_\Theta = f \ast \Xi_\Theta$.
For $1 \le p \le \infty$, we say that the square function estimate ${\mathcal{SQ}}(p \rightarrow p; \alpha)$ holds if the estimate  
\[ 
\|  f \|_p \le C_{\epsilon}\delta^{-\alpha-\epsilon} \| S_\delta f \|_p
\]
holds for all $\epsilon>0$ and all functions $f$ having Fourier support in $\Gamma_\delta$, where $C_\epsilon$ is a positive constant depending on $\epsilon$ but not on $\delta$. 

It was conjectured that the square function estimate ${\mathcal{SQ}}(p \rightarrow p; \alpha)$ holds for $p > 2$ and $\alpha \ge \max(0, \frac{1}{2} - \frac{2}{p})$, see \cite{garrigos2009cone, tao2000bilinearII}. Mockenhaupt  \cite{mockenhaupt1993cone} first considered it, and proved the estimate $\mathcal{SQ}(4 \to 4; 1/8=0.125)$. It was observed by Bourgain \cite{bourgain1995cone}  that the exponent $\alpha$ could be less than $1/8$, and Tao and Vargas \cite{tao2000bilinearII} gave an explicit exponent $\alpha$ by combining their bilinear cone restriction estimates with Bourgain's arguments. After that, the sharp bilinear cone restriction estimate was obtained by Wolff \cite{wolff2001sharp}, and the estimate $\mathcal{SQ}(4 \to 4; 5/44=0.113\dot6\dot3)$ immediately followed by a theorem in \cite{tao2000bilinearII}. 

Garrig\'os and Seeger \cite{garrigos2009cone} have studied \textit{$\ell^p$ decoupling estimates} (called Wolff-type inequalities \cite{wolff2000local}) for cones, and they further improved the exponent $\alpha$ by combining $\ell^p$ decoupling estimates with bilinear arguments in \cite{tao2000bilinearII}. In \cite{wolff2000local}, Wolff  introduced an important type of estimate related to the above square function which have become known as {$\ell^p$ decoupling inequalities}. Decoupling inequalities will play an important role in this paper and will be discussed in detail in section 3. Recently, the sharp $\ell^2$ decoupling theorem for the cone was proved by  Bourgain and Demeter \cite{bourgain2015proof} using the multilinear restriction theorem due to Bennett, Carbery and Tao \cite{bennett2006multilinear}. So, by results in \cite{garrigos2009cone} the estimate $\mathcal{SQ}(4 \to 4; 3/32=0.09375)$ was obtained.
Our first result is to make a further progress on the exponent $\alpha$.

\begin{thm} \label{thm:sqfEst}
The estimate $\mathcal{SQ}(4 \to 4;1/16=0.0625)$ holds.
\end{thm}

The approach to Theorem \ref{thm:sqfEst} is based on trilinear methods.  S. Lee and Vargas \cite{lee2012cone} already employed a trilinear approach to square function estimates by adapting the multilinear arguments of Bourgain and Guth \cite{bourgain2011bounds}, and obtained the sharp estimate $\mathcal{SQ}(3 \to 3;0)$. In \cite{lee2012cone}, it was observed that trilinear square function estimates for the cone are essentially equivalent to linear ones. To get a trilinear square function estimate, the multilinear restriction theorem of Bennet, Carbery and Tao \cite{bennett2006multilinear} will be utilized as in \cite{lee2012cone}. However, to lift the $L^3$ estimate to the $L^4$ estimate we will combine this with the sharp $\ell^2$ decoupling theorem due to Bourgain and Demeter \cite{bourgain2015proof}. Also, we will adapt the induction-on-scales argument of Bourgain and Demeter \cite{bourgain2015proof}. However, since their arguments take advantage of some properties of decoupling norm not derived from the square function, we cannot formulate an iteration as strong as in \cite{bourgain2015proof}.  Nevertheless, it is enough to obtain Theorem \ref{thm:sqfEst}.

The square function estimate is related to several deep questions in harmonic analysis such as the cone multiplier, local smoothing conjecture and the $L^p$ regularity conjecture for convolution operator with the helix. In particular, these conjectures follow from the sharp estimate $\mathcal{SQ}(4 \to 4;0)$, see for example \cite{tao2000bilinearII}, \cite{garrigos2009cone}. 
Theorem \ref{thm:sqfEst} implies the following partial results on these problems.

\begin{cor} \label{cor}
\emph{(i)} If $\alpha > 1/16$ then the local smoothing estimate
\[ 
\Big( \int_{1}^{2} \big\| e^{it \sqrt{-\Delta}} f \big\|^4_{L^4(\mathbb R^2)} dt \Big)^{1/4} \le C_{\alpha} \|f\|_{L_\alpha^{4}(\mathbb R^2)}
\]
holds, where  $L_\alpha^p$ is the $L^p$-Sobolev space of order $\alpha$.

\emph{(ii)} If $\alpha > 1/16$ then the cone multiplier operator  $T_\alpha$ defined by $\widehat{T_\alpha f} (\xi,\tau)=  \rho(\tau) (1-|\xi|^2/ \tau^2)_+^{\alpha}\hat f(\xi)$ is bounded on $L^4$, where $\rho$ is a bump function on $[1,2]$.

\emph{(iii)} If $\alpha < 5/24$ then the convolution operator $T$ defined by\[
Tf(x) = \int f(x_1 -\cos t, x_2 - \sin t, x_3 -t ) \phi(t) dt
\]
maps $L^4$ to $L_{\alpha}^4$, where $\phi$ is a bump function.
\end{cor}

We note that the sharp estimate $L^p \to L^p_{1/p}$, $p >4$, for the averaging operator $T$ may be obtained by combining the theorem due to Pramanik and Seeger \cite{Pramanik2007averages} and the Bourgain--Demeter decoupling estimates.

The proof of Corollary \ref{cor} is well known, and we will not reproduce here, see for example \cite{tao2000bilinearII}. For other related problems, see 
\cite{garrigos2009cone}, \cite{bourgain2015proof}.

\

We are further concerned with $L_\alpha^p \to L^q$ type local smoothing estimates  
\begin{equation} \label{eqn:LS}
\Big( \int_1^2 \big\| e^{it \sqrt{-\Delta}} f \big\|^q_{L^q(\mathbb R^2)} dt \Big)^{1/q} \le C_{p,q,\alpha} \|f\|_{L_\alpha^{p}(\mathbb R^2)}.
\end{equation}
It is conjectured that this local smoothing estimate holds if 
\begin{equation}\label{apq_con}
\begin{gathered} 
1 \le p \le q \le \infty, \\
\frac{1}{p} + \frac{3}{q} = 1, \qquad \alpha \ge \frac{1}{p} -\frac{3}{q} + \frac{1}{2},
\end{gathered}
\end{equation}
see \cite{schlag1997local, tao2000bilinearII}. 
Indeed, the necessity of condition $p \le q$ follows from translation invariance, see \cite{hormander1960estimates}. From the focusing example, Knapp example and delta function, one has three necessary conditions 
\begin{align}
\label{al_1}
\alpha &\ge \frac{1}{p} -\frac{3}{q} + \frac{1}{2},\\
\label{al_2}
\alpha &\ge \frac{3}{2p} - \frac{3}{2q},\\ 
\label{al_3}
\alpha &\ge \frac{2}{p} -\frac{1}{q} - \frac{1}{2},
\end{align}
respectively, see \cite{tao2000bilinearII} for details.
Let $I_1 = (1,1;1/2+\varepsilon),~ I_2 =(2,2;0),~ I_\infty =(\infty, \infty; 1/2+\varepsilon),~ I_{1,\infty} =(1,\infty;3/2+\varepsilon)$ where $\varepsilon >0$ is arbitrary. When $(p,q;\alpha) = I_1, I_2, I_\infty$ and $I_{1,\infty}$, one can obtain \eqref{eqn:LS} from the fixed-time estimates due to Miyachi \cite{miyachi1980some} and Peral \cite{peral1980lp}. First, in case that \eqref{al_3} is dominant, the reciprocal range $(1/p,1/q)$ is the triangular shape with vertices $(1,1)$, $(1/2,1/2)$ and $(1,0)$. In this case, by interpolation, the estimates \eqref{eqn:LS} for such triangular shape range follow from the estimates for $I_1, I_2$ and $I_{1, \infty}$. We see that the conjecture \eqref{apq_con} satisfies both \eqref{al_1} and \eqref{al_2}. If we have the conjecture, by interpolating between \eqref{apq_con} and $I_{\infty}$ 
the estimates \eqref{eqn:LS} are obtained when \eqref{al_1} is dominant, and analogously the interpolation between \eqref{apq_con} and $I_{2}$ gives the estimates \eqref{eqn:LS} when \eqref{al_2} is dominant. For an endpoint $(p,q;\alpha)=(4,4;0)$, it is known that the local smoothing estimate does not hold, see \cite{Wolff96recentwork}. But, for $q > 4$, $\frac{1}{p} + \frac{3}{q} =1$ and $\alpha = \frac{1}{p} - \frac{3}{q} + \frac{1}{2}$, it is not known whether the local smoothing estimate holds or not. 

The critical $L_\alpha^4 \to L^4$ estimate has been considered in Corollary \ref{cor}. We continue to study a sharp $L_\alpha^p \to L^q$ estimate when $p<q$. From Strichartz' estimate $L_{1/2}^2 \to L^6$, this conjecture follows for $q \ge 6$. Schlag and Sogge \cite{schlag1997local} first improved this to $q \ge 5$, and Tao and Vargas \cite{tao2000bilinearII} made further progress by using bilinear approach. By the sharp bilinear cone restriction estimate due to Wolff \cite{wolff2001sharp} and the results in \cite{tao2000bilinearII}, the conjecture was improved to $q \ge 14/3 = 4.\dot6$, and the $\epsilon$-loss of $\alpha$ was removed by S. Lee \cite{lee2003endpoint}. 
Our second result is to obtain an improved sharp local smoothing estimate.

\begin{thm} \label{thm:localSm}
The estimate \eqref{eqn:LS} holds  
for $q \ge 30/7=4.\dot28571\dot4$ and $p,\alpha$ satisfying the conditions in \eqref{apq_con} except the endpoint $(p,q;\alpha) = (10/3, 30/7 ;1/10)$.
\end{thm}

Theorem \ref{thm:localSm} will be proved through the trilinear approach too. The proof is simpler than Theorem \ref{thm:sqfEst}. We will reduce this linear estimate to a trilinear one, and the desired trilinear estimate will be obtained from interpolating between two trilinear estimates deduced from the multilinear restriction theorem \cite{bennett2006multilinear} and the $\ell^2$ decoupling theorem \cite{bourgain2015proof}.  

\

Throughout this paper, 
we write $A \lesssim B$ or $A = O(B)$ if $A \le CB$ for some constant $C >0$ which may depend on $p$, $q$ but not on $\delta$, $R$ and $N$, and $A \sim B$ if $A\lesssim B$ and $B \lesssim A$.  
The constants $C$, $C_\varepsilon$, $C_{\epsilon}$, $C_{\epsilon_1}$ and the implicit constants in $\lesssim$ and $\sim$ will be adjusted numerous times throughout the paper.
For any finite set $A$, we use $\#A$ to denote its cardinality, and
if $A$ is a measurable set, we use $|A|$ to denote its Lebesgue
measure. 
If $R$ is a rectangular box or an ellipsoid and $k$ is a positive real number, we use $kR$ to denote the $k$-dilation of $R$ with center of dilation at the center of $R$.

\section{Reduction to a trilinear estimate}

In this section, we will show that the linear square function estimate is equivalent to a trilinear one.
The arguments of this section are a small modification of arguments found in \cite{lee2012cone}. Specifically, we replace $L^3$ arguments by $L^p$ ones for $p \ge 2$.

For an arc $\Omega \subset S^1$ we define a sector $\Gamma^{\Omega}$ and a $\delta$-fattened sector $\Gamma_\delta^{\Omega}$ by
\[
\Gamma^\Omega = \{ (\xi,\tau) \in \Gamma : \xi/|\xi| \in \Omega \},
\qquad
\Gamma_\delta^\Omega = \{ (\xi,\tau) \in \Gamma_\delta : \xi/|\xi| \in \Omega \}.
\] 
Let $\Omega_1, \Omega_2, \Omega_3  \subset S^1$ be arcs whose lengths are comparable to each other. We say that  $\Gamma^{\Omega_1}, \Gamma^{\Omega_2}, \Gamma^{\Omega_3}$ are \textit{$\nu$-transverse} if for any unit normal vector $n_i$ to $\Gamma^{\Omega_i}$, $i=1,2,3$, the parallelepiped formed by $n_1, n_2, n_3$ has volume $\ge \nu$, see Figure \ref{fig:transversal}. A key geometric property of the cone $\Gamma$ is that $\Gamma^{\Omega_1}, \Gamma^{\Omega_2}, \Gamma^{\Omega_3}$ are $\nu$-transverse if and only if $\Omega_1, \Omega_2, \Omega_3$ are mutually separated by a distance $\gtrsim \nu^{1/3}$, see \cite{lee2012cone}.

\begin{figure}
\begin{center}
\begin{tikzpicture}[scale = 1.7]
	\draw[semithick] (4,1.5) ellipse (2 and 0.2);	
	\draw[semithick] (2.5,0) arc (180:360:1.5 and 0.2); 
	\draw[dashed,color=gray] (2.5,0) arc (180:0:1.5 and 0.2); 
	\draw[semithick] (2.502,-0.01) -- (2,1.5);
	\draw[semithick] (5.498,-0.01) -- (6,1.5);
	\draw[semithick] (3.3,-0.175) -- (2.85,1.34);
	\draw[semithick] (3.9,-0.2) -- (3.7,1.3);
	\draw[->,semithick,color=gray] (3.3,1.1) -- (3.6,1.6);
	\draw (3.9,1.45) node {$\Gamma^{\Omega_2}$};
	\draw[semithick] (3.1,1.68) -- (3.2,1.32);
	\draw[dashed,color=gray] (3.2,1.32) -- (3.5,0.2);
	\draw[semithick] (2.3,1.605) -- (2.4,1.375);
	\draw[dashed,color=gray] (2.4,1.375) -- (2.88,0.15);
	\draw[->,semithick,color=gray] (2.8,1.45) -- (3.2,1.8);
	\draw (2.9,1.8) node {$\Gamma^{\Omega_1}$};
	\draw[semithick] (5.2,-0.12) -- (5.5,1.37);
	\draw[semithick] (5.7,1.4)-- (5.75,1.6);
	\draw[dashed,color=gray] (5.7,1.4) -- (5.4,0.07);
	\draw[->,semithick,color=gray] (5.8,1.25) -- (5.4,1.6);
	\draw (5.8,1.75) node {$\Gamma^{\Omega_3}$};
\end{tikzpicture}
\end{center}
\caption{}
\label{fig:transversal}
\end{figure}
%
%

Let us use the notation $\mathcal{SQ}(p \times p \times p \rightarrow p;\alpha)$ if one has the trilinear square function estimate 
\[ 
\Big\| \Big( \prod_{i=1}^{3} |f_i| \Big)^{1/3} \Big\|_p
\le C_{\nu,\epsilon} \delta^{-\alpha-\epsilon} \Big( \prod_{i=1}^{3} \|S_\delta f_i \|_p \Big)^{1/3}
\]
for all $\epsilon>0$ and all $f_i$ with $\supp \hat f_i \subset \Gamma_\delta^{\Omega_i}$, where $\Omega_1, \Omega_2, \Omega_3$ are any arcs such that their lengths are comparable to each other, and $\Gamma^{\Omega_1}, \Gamma^{\Omega_2}, \Gamma^{\Omega_3}$ are $\nu$-transverse. 
It is easy to see that $\mathcal{SQ}(p \rightarrow p;\alpha)$ implies $\mathcal{SQ}(p \times p \times p \rightarrow p;\alpha)$ by H\"older's inequality. We will show that the converse is true.
Let \( 1 > \gamma_1 > \gamma_2 > 0 \) be small positive numbers.  
We define \( \mathbf \Omega(\gamma) \) to be a family of $O(\gamma^{-1})$ arcs of length $\gamma$ covering the unit circle with finite overlap. 
We take a Schwartz function $\Xi_\Omega$  whose Fourier transform $\widehat \Xi_\Omega$ is a bump function supported on a neighborhood of  $\Gamma_\delta^\Omega$.
The following is due to S. Lee and Vargas  \cite{lee2012cone}*{equation (23)}.

\begin{lem}[Lee--Vargas  \cite{lee2012cone}*{equation (23)}] \label{lem:LinTriLcompare}
Suppose that $f$ has Fourier support in $\Gamma_\delta$ and let $0 < \gamma_2 < \gamma_1 < 1$.
Then for any $x \in \mathbb R^3$, 
\begin{align} 
| f(x) | &\lesssim \max_{\Omega \in \mathbf 
\Omega(\gamma_1)}| f_{\Omega}(x) | 	
+ \gamma_1^{-1} 
\max_{\Omega \in \mathbf \Omega(\gamma_2)}	| f_{\Omega}(x)|  
+ \gamma_2^{-50} 
\max_{\substack{\Omega_1,\Omega_2,\Omega_3 \in \mathbf 
\Omega(\gamma_2): \\ \dist(\Omega_i, \Omega_j) \ge \gamma_2,\, i 
\neq j}} \Big( \prod_{i=1}^{3} |f_{\Omega_i}(x)| \Big)^{1/3}
\end{align}
where $f_\Omega = f \ast \Xi_\Omega$.
\end{lem}

To obtain the above lemma, S. Lee and Vargas adapted the arguments of Bourgain and Guth \cite{bourgain2011bounds} who made progress on the restriction conjecture by using a multilinear approach. 
 
Using Lemma \ref{lem:LinTriLcompare} we can establish the following relation between the linear and trilinear square function estimates.

\begin{prop} \label{prop:MSQmeanSQ}
Let $p \ge 2$ and $\alpha \ge 0$. Suppose that $\mathcal{SQ}(p\times p\times p \rightarrow p;\alpha)$ holds. Then $\mathcal{SQ}(p \rightarrow p;\alpha)$ is valid. 
\end{prop}

\begin{proof}
Let $\epsilon>0$ be given. We assume that $\beta \ge 0$ is the best exponent for which 
\begin{equation} \label{asshy}
\|  f \|_p \le C \delta^{-\beta-\epsilon} \| S_\delta f \|_p
\end{equation}
holds for all $f$ with $\supp \hat f \subset \Gamma_\delta$, i.e., 
\[ 
\beta = \inf_{\delta > 0} \Big( \log_{1/\delta} \sup_{f: \supp \hat f \subset \Gamma_\delta}  \frac{ \| f \|_p } {\| S_\delta f \|_p} \Big) - \epsilon.
\]

It suffices to show that for any small $0<\epsilon_1<1$,
\begin{equation} \label{expRel}
\beta \le \alpha+ O(\epsilon_1)+\log_{1/\delta} C_{\epsilon, \epsilon_1},
\end{equation}
since if we choose a sufficiently small $\epsilon_1$ then  $O(\epsilon_1)$ is bounded by $\epsilon$, which can be absorbed in an $\epsilon$-loss in the estimate $\mathcal{SQ}(p \to p;\alpha)$. 
The dependence on $\epsilon$ and $\epsilon_1$ of the constant $C_{\epsilon, \epsilon_1}$ in the above inequality comes from employing $\mathcal{SQ}(p \times p \times p \to p; \alpha)$. Especially $\epsilon_1$ is related to the transversality of trilinear estimates below.

We may assume that $\delta>0$ is sufficiently small, say $0< \delta \le \delta_0$, because the desired estimate is trivially obtained, otherwise, where $\delta_0$ is a small parameter to be fixed later in the proof.  Let \( 1 > \gamma_1  > \gamma_2\ge \delta_0^{\epsilon_1/2} \) be dyadic multiples of $\delta^{1/2}$, the value of which is to be fixed later in the argument. 
By Lemma \ref{lem:LinTriLcompare} and the embedding $\ell^{p} \subset \ell^{\infty}$,
\begin{equation}  \label{LpTricom}
\begin{split}
\| f \|_p^p &\lesssim  \sum_{\Omega_1 \in \mathbf 
\Omega(\gamma_1)} \| f_{\Omega_1} \|_p^p
+  \gamma_1^{-p} 
\sum_{\Omega_2 \in \mathbf \Omega(\gamma_2)}	\| f_{\Omega_2} \|_p^p \\
&\qquad\qquad + \gamma_2^{-50p} 
\sum_{\substack{\Omega_1,\Omega_2,\Omega_3 \in \mathbf 
\Omega(\gamma_2): \\ \dist(\Omega_i, \Omega_j) \ge \gamma_2,\, i 
\neq j}} \Big\| \Big( \prod_{i=1}^{3} |f_{\Omega_i}| \Big)^{1/3} \Big\|_p^p,
\end{split}
\end{equation}
where $\Omega_j$ is taken such that if $\theta$ intersects the interior of $\Omega_j$ then $\theta \subset \Omega_j$ for $j=1,2$.

Consider the first and second summation in the right-hand side of \eqref{LpTricom}. For convenience we denote by $\Omega = \Omega_j$ and $\gamma =\gamma_j$. Using Lorentz rescaling we will show 
\begin{equation} \label{ppsc}
\| f_\Omega \|_p \le C_\epsilon   (\delta / \gamma^2)^{-\beta-\epsilon} \| S_\delta f_{\Omega} \|_p.
\end{equation}
By rotating the unit circle we may assume that $\Omega$ is centered at $(1,0)$. Let $T : \mathbb R^3 \to \mathbb R^3$ be a linear transformation so that 
\[
T(e_1,1) = (e_1,1),\quad  T(-e_1,1) = \gamma^{2}(-e_1,1),\quad T(e_2,0) = \gamma(e_2,0)
\]
where $\{e_1, e_2\}$ is a standard basis in $\mathbb R^2$. Then $\hat f_\Omega \circ T$ is supported in $\Gamma_{\delta/\gamma^2}$.
From the equation $\widehat{f_\Omega \circ T^{-t}} = |\det T| \hat f_\Omega \circ T$, it follows that $\widehat{f_\Omega \circ T^{-t}}$ has support in $\Gamma_{\delta/\gamma^2}$ where $T^{-t}$ is the inverse transpose of $T$. Since $\gamma \ge \delta^{1/2}$, by \eqref{asshy} it follows that 
\begin{equation} \label{bsc}
\|f_\Omega \circ T^{-t}\|_p \lesssim (\delta/\gamma^2)^{-\beta-\epsilon} \|S_{\delta/\gamma^2} (f_\Omega \circ T^{-t}) \|_p.
\end{equation}
By definition,
\[
S_{\delta/\gamma^2} (f_\Omega \circ T^{-t}) = \Big( \sum_{\Upsilon \in \mathbf \Pi_{\delta/\gamma^2}} \big|(f_\Omega \circ T^{-t}) \ast \Xi_{\Upsilon} \big|^2 \Big)^{1/2}.
\]
From $\ \hat \Xi_\Upsilon \circ T^{-1}= \hat \Xi_{T(\Upsilon)}$, it follows that
\(
\big((f_\Omega  \circ T^{-t}) \ast \Xi_{\Upsilon} \big)\sphat =|\det T|(\hat f_\Omega  \circ T)  \hat \Xi_{\Upsilon,} 
= |\det T|(\hat f_\Omega  \hat \Xi_{T(\Upsilon)} ) \circ T.
\)
Thus, by taking the inverse Fourier transform,
\[
(f_\Omega  \circ T^{-t}) \ast \Xi_{\Upsilon} = (f_\Omega \ast \Xi_{T(\Upsilon)}) \circ T^{-t}.
\]
Since $f_\Omega \ast \Xi_{T(\Upsilon)}$ has Fourier support in $T(\Upsilon)$ which is a sector of size $1 \times \delta \times C\delta^{1/2}$ in $\Gamma_{\delta}$, we have
\[
S_{\delta/\gamma^2} (f_\Omega \circ T^{-t}) = \Big( \sum_{\Upsilon \in \mathbf \Pi_{\delta/\gamma^2}} |(f_\Omega \ast \Xi_{T(\Upsilon)}) \circ T^{-t} |^2 \Big)^{1/2} =  (S_{\delta} f_\Omega) \circ T^{-t}.
\]
We substitute this in \eqref{bsc} and remove $T^{-t}$ by changing variables. Then we obtain \eqref{ppsc}.

By \eqref{ppsc} we have
\[
\sum_{\Omega \in \mathbf \Omega(\gamma)} \| f_{\Omega} \|_p^p 	\le C_\epsilon (\delta /\gamma^{2})^{-p\beta-p\epsilon} \sum_{\Omega \in \mathbf \Omega(\gamma)} \| S_\delta f_{\Omega} \|_p^p. 
\]
Since we can decompose $f_\Omega = \sum_{\Theta \in \mathbf \Pi_\delta: \theta \subset \Omega} f \ast \Xi_\Theta$, we have that for $p \ge 2$, 
\begin{align*}
\sum_{\Omega \in \mathbf \Omega(\gamma)} \| S_\delta f_{\Omega} \|_p^p &= \sum_{\Omega \in \mathbf \Omega(\gamma)} \int \Big( \sum_{\Theta \in \mathbf \Pi_\delta: \theta \subset \Omega} | f \ast \Xi_\Theta|^2 \Big)^{p/2} \\
&\le \int \Big(\sum_{\Omega \in \mathbf \Omega(\gamma)}  \sum_{\Theta \in \mathbf \Pi_\delta : \theta \subset \Omega} | f \ast \Xi_\Theta|^2 \Big)^{p/2} \\
&\le \|S_\delta f\|_p^p.
\end{align*}
Inserting this into the previous estimate, we obtain
\begin{equation} \label{indP}
\sum_{\Omega \in \mathbf \Omega(\gamma)} \| f_{\Omega} \|_p^p 	\le C_\epsilon (\delta /\gamma^{2})^{-p\beta-p\epsilon} \|S_\delta f\|_p^p.
\end{equation}

Consider the trilinear part in \eqref{LpTricom}. By applying $\mathcal{SQ}(p\times p \times p \rightarrow p; \alpha)$,
\begin{equation} \label{Tripart}
\sum_{\substack{\Omega_1,\Omega_2,\Omega_3 \in \mathbf 
\Omega(\gamma_2): \\ \dist(\Omega_i, \Omega_j) \ge \gamma_2,\, i 
\neq j}} \Big\| \Big( \prod_{i=1}^{3} |f_{\Omega_i}| \Big)^{1/3} \Big\|_p^p \le C_{\epsilon, \gamma_2} \gamma_2^{-3} \delta^{-p\alpha-p\epsilon} \|S_\delta f \|_p^p.
\end{equation}
We substitute \eqref{indP} and \eqref{Tripart} in \eqref{LpTricom}. Then,
\[ 
\|f\|_p \le (C_\epsilon \gamma_1^{2(\beta+\epsilon)} \delta^{-\beta-\epsilon} + C_\epsilon \gamma_1^{-1} \gamma_{2}^{2(\beta+\epsilon)} \delta^{-\beta-\epsilon} + C_{\epsilon, \gamma_2}\gamma_2^{-60} \delta^{-\alpha-\epsilon}) \|S_\delta f \|_p.
\]
So, by the assumption for $\beta$,
\begin{align*} 
\delta^{-\beta} &\le (C_\epsilon \gamma_1^{2(\beta+\epsilon)} + C_\epsilon \gamma_1^{-1} \gamma_{2}^{2(\beta+\epsilon)} )\delta^{-\beta} + C_{\epsilon,\gamma_2}\gamma_2^{-60} \delta^{-\alpha}. 
\end{align*}
We now choose $\gamma_1, \gamma_2$ and $\delta_0$ so that   $C_\epsilon\gamma_1^{2(\beta+\epsilon)} \le 1/4$, $ C_\epsilon \gamma_1^{-1} \gamma_{2}^{2(\beta+\epsilon)}  \le 1/4$ and $1> \gamma_1 > \gamma_2 \ge \delta_0^{{\epsilon_1}/{2}}$.
Then $\delta^{-\beta} \le C_{\epsilon,\gamma_2} \gamma_2^{-60}\delta^{-\alpha} \le  C_{\epsilon, \epsilon_1} \delta^{-30\epsilon_1 - \alpha}$, which means \eqref{expRel}. 
\end{proof}

\section{Decoupling norms} \label{sec:decoupling}
In this section, we will show that the decoupling norm for the cone essentially satisfies the reverse H\"older inequality, and apply this to the interpolation between decoupling estimates. In fact, our interpolation lemmas can be obtained by using known interpolation theorems, so our proof is an alternative one (which is actually weaker).  
This section is obtained by modifying the arguments for paraboloid decoupling in \cite{bourgain2015proof}*{section 3}.  For further discussion for decoupling, see \cite{wolff2000local}, \cite{laba2002local}, \cite{garrigos2008improvements}, \cite{garrigos2010mixed}. 

Let $f$ be a function having Fourier support in $\Gamma_\delta$.
For such functions, the norm $\| \cdot \|_{p,\delta}$, $1 \le p \le \infty$ is defined by
\[ 
\|f\|_{p,\delta} := \Big( \sum_{\Theta \in \mathbf \Pi_{\delta}} \| f_{\Theta} \|^2_{p} \Big)^{1/2}.
\]
It is easy to see that if $m$ is a positive real number then $\|f\|_{p,m\delta} \le C_m \|f\|_{p,\delta}$ by Minkowski's inequality.

We first introduce a wave packet decomposition, which is a fundamental tool for studying Fourier restriction type problems. To decompose $f$ both in frequency space and in spatial space, we define standard bump functions.
Let $\phi(x) := (1+|x|^2)^{-M/2}$ where $M$ is a sufficiently large exponent. Let $\psi: \mathbb R^{3}\rightarrow \mathbb R$ be a nonnegative Schwartz function such that $\psi$ is strictly positive in the unit ball $B(0,1)$, Fourier supported in a ball $B(0,1/4)$ and $\sum_{k \in \mathbb Z^3}\psi(x-k) = 1$. For an ellipsoid $E$, we define $a_E$ to be an affine map from the unit ball $B(0,1)$ to $E$. Let $\phi_E = \phi \circ a_E^{-1}$ and $\psi_E = \psi \circ a_E^{-1}$.

\begin{lem} \label{lem:wavepack}
Suppose that $f$ is Fourier supported in $\Gamma_\delta$. Then there exists a decomposition
\begin{equation} \label{waveDec}
f(x) = \sum_{\Theta \in \mathbf \Pi_\delta}\sum_{\pi \in \mathbf P_\Theta} h_{\pi} f_{\pi}(x),
\end{equation}
where $\mathbf P_\Theta = \mathbf P_\Theta(f)$ is a family of separated rectangles $\pi$ of size $\delta^{-1} \times \delta^{-1/2} \times 1$ with its dual $\pi^* = \Theta$, such that the coefficients $h_\pi > 0$ have the property that 
\begin{equation} \label{lp_sum}
\Big( \sum_{\Theta \in \mathbf \Pi_\delta}\Big( \sum_{\pi \in \mathbf P_\Theta} |\pi| h_\pi^p \Big)^{2/p} \Big)^{1/2} \sim \|f\|_{p,\delta}
\end{equation}
for all $1 \le p < \infty$ and
\begin{equation} \label{l_infty}
\Big( \sum_{\Theta \in \mathbf \Pi_\delta} \sup_{\pi \in \mathbf P_\Theta} h_\pi^2 \Big)^{1/2} \sim \|f\|_{\infty,\delta},
\end{equation}
and the functions $f_\pi$ obey
\begin{equation} \label{fourierSupp}
\supp \hat f_\pi \subset 4\Theta
\end{equation}
and
\begin{equation} \label{ess_supp}
  |f_{\pi}(x)| \lesssim \phi_\pi(x).
\end{equation}
\end{lem}

\begin{proof}
For each $\Theta \in \mathbf \Pi_\delta$, we partition $\mathbb R^3$ into the dual rectangles $\pi$ of $\Theta$. 
For each $\pi$, we define a coefficient $h_\pi$ and a function $f_\pi$ by
\[
  h_{\pi} = \frac{1}{|\pi|}\int |f_\Theta(x)| \psi_\pi(x) dx \qquad \text{and} \qquad
  f_\pi(x) = h_\pi^{-1} \psi_\pi(x) f_\Theta (x).
\]
Then, \eqref{fourierSupp} immediately follows, and some direct calculating gives \eqref{waveDec}.  
By Bernstein's inequality, 
\begin{equation*}
 |\psi_{\pi}(x) f_\Theta (x)| \lesssim h_\pi,
\end{equation*}
so we have $|f_\pi(x)| \lesssim |\psi_\pi(x)|$. This implies \eqref{ess_supp}.

By H\"older's inequality we have $h_\pi \lesssim \Big( \frac{1}{|\pi|}\int |f_\Theta(x)|^p \psi_\pi(x) dx \Big)^{1/p}$, and using Bernstein's lemma we can see that $\Big( \frac{1}{|\pi|}\int |f_\Theta(x)|^p \psi_\pi(x) dx \Big)^{1/p}  \lesssim h_\pi$. So, we have
\[
\sum_{\pi \in \mathbf P_\Theta}  |\pi| h_\pi^p \sim \sum_{\pi \in \mathbf P_\Theta} \int |f_\Theta|^p \psi_\pi 
= \|f_\Theta\|_p^p,
\]
from which \eqref{lp_sum} follows. Similarly, we have that $h_\pi \sim \sup_{x \in \pi} |f_\Theta(x)|$ and that $\sup_{\Theta \in \mathbf P_\Theta} h_\pi \sim \|f_\Theta\|_\infty$. Thus \eqref{l_infty} follows. 
%
\end{proof}

Now we study the reverse H\"older inequality for the decoupling norm. 
We say that $f$ is a \textit{balanced function} if $f$ is a function of the form \eqref{waveDec} with $h_\pi=1$
such that $f$ satisfies \eqref{fourierSupp}, \eqref{ess_supp}
and a property that for any $\Theta, \Theta' \in \mathbf \Pi_\delta$, the nonempty $\mathbf P_{\Theta}(f), \mathbf P_{\Theta'}(f)$ have comparable cardinality. These kinds of functions were first explicitly used by Wolff \cite{wolff2000local}.

\begin{lem} \label{lem:revHol}
Suppose that $1 \le p,q,r \le \infty$ and that for some $\theta \in (0,1)$, 
\[
\frac{1}{r} = \frac{1-\theta}{q} + \frac{\theta}{p}.
\]
Then
\[ 
\| f \|_{r,\delta} \sim \|f\|_{q, \delta}^{1-\theta} \|f\|_{p, \delta}^{\theta},
\]
for all balanced function $f$. 
\end{lem}

\begin{proof}
Since $f$ is a balanced function, there is a number $\kappa>0$ such that every nonempty $\mathbf P_\Theta(f)$ has cardinality comparable to $\kappa$. Let $\nu$ be the number of nonempty $\mathbf P_\Theta(f)$. Then by \eqref{lp_sum} and \eqref{l_infty}, one has
\[ 
\|f\|_{r,\delta} \sim \nu^{1/2} \kappa^{1/r} |\pi|^{1/r}
= \nu^{\frac{1-\theta}{2}} \kappa^{\frac{1-\theta}{q}} |\pi|^{\frac{1-\theta}{q}} \nu^{\frac{\theta}{2}} \kappa^{\frac{\theta}{p}} |\pi|^{\frac{\theta}{p}} \sim \|f\|_{q, \delta}^{1-\theta} \|f\|_{p, \delta}^{\theta}.
\]
\end{proof}

As an application we have the following interpolation lemma.
 
\begin{lem} \label{lem:interp}
Let $2 \le p_1, p_2, q_1, q_2 \le \infty$.
Assume that
\begin{equation} \label{givenEstLL}
\| f \|_{q_1} \le A_1 \| f \|_{p_1,\delta},\qquad
\| f \|_{q_2} \le A_2 \| f \|_{p_2,\delta}
\end{equation}
for all $f$ with $\supp \hat f \subset \Gamma_\delta$.
Suppose that for some $\theta \in (0,1)$,
\[ 
\frac{1}{q} = \frac{1-\theta}{q_1} + \frac{\theta}{q_2}, \qquad  
\frac{1}{p} = \frac{1-\theta}{p_1} + \frac{\theta}{p_2},
\]
and $2 \le p \le q \le \infty$.
Then  
\begin{equation} \label{ineq:intp}
\| f \|_{q} \lesssim \delta^{-\varepsilon}A_1^{1-\theta} A_2^{\theta}  \| f \|_{p,\delta}
\end{equation}
for all $f$ with $\supp \hat f \subset \Gamma_\delta$ and all $\varepsilon >0$.
\end{lem}



\begin{proof}
For localization we decompose $f = \sum_{k \in \delta^{-1} \mathbb  Z^3} \psi_k f$ where $\psi_k := \psi(\delta(x-k))$. Then,
\[
\|f\|_q^q \le \sum_{k' \in \delta^{-1} \mathbb  Z^3} \Big\| \sum_{k \in \delta^{-1} \mathbb  Z^3} \psi_k f \Big\|_{L^q(B(k',2\delta^{-1}))}^q.
\]
Since $\psi_k$ has rapid decay outside $B(k,\delta^{-1-\varepsilon})$, we have that if $x \in B(k',2\delta^{-1})$ then 
\[
\Big|\sum_{k \in \delta^{-1}\mathbb Z^3 \setminus B(k',2\delta^{-1-\varepsilon})} \psi_k(x) \Big| \le C_K \delta^{K}
\] 
for all $K>0$. Using this and a rough estimate $\|f\|_q \lesssim \delta^{-C} \|f\|_{p,\delta}$, we have that for any $\varepsilon>0$ and $K>0$,
\[
\|f\|_q^q \le \sum_{k'} \Big\| \sum_{k \sim k'} \psi_k f \Big\|_{L^q(B(k',2\delta^{-1}))}^q + C_K\delta^{K}\|f\|_{p,\delta}^q,
\]
%
where $k\sim k'$ means that $k \in B(k',2\delta^{-1-\varepsilon}) \cap \delta^{-1}\mathbb Z^{3}$. Since the number of $k \in \delta^{-1}\mathbb Z^3$ contained in $B(k',2\delta^{-1-\varepsilon})$ is $O(\delta^{3\varepsilon})$, we have 
\begin{align*}
\|f\|_q^q  &\lesssim \delta^{-3\varepsilon q} \sum_{k'} \sum_{k \sim k'}  \|\psi_k f \|_{L^q(B(k',2\delta^{-1})}^q + C_K\delta^{K} \|f\|_{p,\delta}^q\\
&\lesssim \delta^{-3\varepsilon q} \sum_{k'} \sum_{k \sim k'}  \|\psi_k f \|_q^q + C_K\delta^{K} \|f\|_{p,\delta}^q \\
&\lesssim \delta^{-3\varepsilon q-3\varepsilon} \sum_{k}  \|\psi_k f \|_q^q + C_K\delta^{K}\|f\|_{p,\delta}^q.
\end{align*}
Since $p \le q$, we have that for any $\varepsilon >0$ and any $K > 0$,
\[
\|f\|_q \lesssim \delta^{-C\varepsilon}\Big( \sum_{k} \| \psi_k f \|_q^{p} \Big)^{1/p} + C_K\delta^{K}\|f\|_{p,\delta}.
\]

On the other hands, by Minkowski's inequality and $p \ge 2$ it follows that
\[
 \Big( \sum_{k} \| \psi_k f \|_{p,2\delta}^{p} \Big)^{1/p} \le \|f\|_{p,2\delta} \lesssim \|f\|_{p,\delta}.
\]
Thus, by the above two estimates  the proof of \eqref{ineq:intp} is reduced to showing 
\begin{equation*}
 \| \psi_k f \|_q \lesssim \delta^{-\varepsilon} A_1^{1-\theta} A_2^{\theta}  \| \psi_k f \|_{p,2\delta}.
\end{equation*}
By translation invariance it is enough to consider $\psi_0 f$. Let $g := \psi_0 f$.
By normalization we may assume that $\|g\|_{p,2\delta} = 1$. Then it is reduced to showing
\begin{equation} \label{redForm}
\|g\|_q  \lesssim \delta^{-\varepsilon} A_1^{1-\theta} A_2^{\theta}.  
\end{equation}
Since $\psi_0$ has fast decay outside $B(0,C\delta^{-1})$, 
we have
\(
\|g\|_q \le \|g\|_{L^q(B(0,\delta^{-1-\varepsilon}))} + C_K\delta^{K}
\)
for all $\varepsilon>0$ and $K>0$. 
Since $\psi_0$ has Fourier support in $B(0,\delta/2)$,  $\widehat g$ is supported in  $\Gamma_{2\delta}$.
By Lemma \ref{lem:wavepack}, it is decomposed into
\[
g(x) = \sum_{\Theta \in \mathbf \Pi_{2\delta}}\sum_{\pi \in \mathbf P_\Theta} h_{\pi} g_{\pi}(x).
\]
We first remove some minor $\pi$'s. By \eqref{ess_supp}, we can eliminate $\pi$ that is disjoint from $B(0, C\delta^{-1-\varepsilon})$.  Let $\mathring{\mathbf P}$ be the collection of $\pi$ intersecting $B(0, C\delta^{-1-\varepsilon})$. Then $\# \mathring{\mathbf P} \lesssim \delta^{-2-3\varepsilon}$.
The rectangles $\pi$ with $h_\pi = O(\delta^{500})$ can be also eliminated, since 
\[
\Big\| \sum_{\pi \in \mathring{\mathbf P} : 0< h_\pi  \lesssim \delta^{500}} h_\pi g_\pi \Big\|_q
\lesssim \delta^{500} |\pi| \# \mathring{\mathbf P}  \lesssim \delta^{400}.
\]

We group the rectangles $\pi$ by value of coefficients $h_\pi$. Since $\|g\|_{p,2\delta} =1$, from \eqref{lp_sum} we can see that $h_\pi \lesssim 1$. 
For any dyadic number $\delta^{500} \lesssim h \lesssim 1$ we define  
\(
\mathring{\mathbf P}_h := \{ \pi \in \mathring{\mathbf P}: h \le h_\pi < 2h  \}. 
\)
It is classified into $\mathring{\mathbf P}_{h, \Theta} := \mathring{\mathbf P}_h \cap \mathbf P_\Theta$, and let
\[ 
\mathring{\mathbf P}_{h}^k := \bigcup_{k \le \# \mathring{\mathbf P}_{h, \Theta} < 2k} \mathring{\mathbf P}_{h, \Theta}
\]
for dyadic numbers $1 \le k \lesssim \delta^{-2}$. 
Since there are  $O(\log \delta^{-1})$ dyadic numbers $\delta^{500} \lesssim h \lesssim 1$ and $1 \le k \lesssim \delta^{-2}$, by pigeonholing	 there exist $h$ and $k$ so that
\[ 
\Big\|
\sum_{\delta^{500} \le h \lesssim 1} h \sum_{1 \le k \lesssim \delta^{-2}} \sum_{\pi \in \mathring{\mathbf P}_h^k}  g_\pi \Big\|_q \lesssim (\log \delta^{-1})^2 h \Big\|\sum_{\pi \in \mathring{\mathbf P}_h^k}  g_\pi \Big\|_q .
\]
Let $\tilde g := \sum_{\pi \in \mathring{\mathbf P}_h^k}  g_\pi $. Then from these estimates, one has
\[
\|g\|_q \lesssim \delta^{-\varepsilon} h \|\tilde g\|_q +\delta^{400}.
\]
Since $\tilde g$ is a balanced function, from H\"oler's inequality, \eqref{givenEstLL} and Lemma \ref{lem:revHol} it follows that
\[ 
\|\tilde g\|_q \le \|\tilde g\|_{q_1}^{1-\theta} \|\tilde g\|_{q_2}^{\theta}
\le A_1^{1-\theta} A_2^{\theta}\|\tilde g\|_{p_1,2\delta}^{1-\theta}\|\tilde g\|_{p_2,2\delta}^{\theta} \lesssim A_1^{1-\theta} A_2^{\theta} \|\tilde g\|_{p,2\delta},
\]
and by \eqref{lp_sum},
\[
h\| \tilde g \|_{p,2\delta} \lesssim \| g \|_{p,2\delta}.
\]
Therefore, by combining these estimates we obtain \eqref{redForm}.
\end{proof}

\begin{remark} \label{rem:Alt_inter}
By using known interpolation theorems we can obtain Lemma \ref{lem:interp} without $\varepsilon$-losses. Indeed, since $f$ in Lemma \ref{lem:interp} has the Fourier support condition, we are not able to apply interpolation theorems directly. To avoid this, 
we define a linear operator $T$ by
\[
T \mathbf f = \sum_{j \in J} f_j \ast \Xi_{\Theta_j}
\] 
for $\mathbf f = \{ f_j\}_{j \in J}$, where $J$ is an index set of $\mathbf\Pi_\delta$. Then the inequality
\(
\|f\|_q \le A \|f\|_{p,\delta}
\)
in Lemma \ref{lem:interp}
is equivalent to
\(
\| T \mathbf f \|_q \le A \| \mathbf f\|_{\ell^2(L^p)},
\) where $\ell^2(L^p)$ is the space of $L^p$-valued $\ell^2$-sequences.
Since the functions $\{f_j\}_{j \in J}$ are not subject to the Fourier support condition, by applying the complex interpolation theorem we get  Lemma \ref{lem:interp} without $\varepsilon$-losses.
\end{remark}

\

To prove Theorem \ref{thm:sqfEst} we need a trilinear interpolation lemma. Before stating the lemma let us define a notation  $\underline\prod$, which will be repeatedly used in the remaining parts of this paper. For $A_1, A_2, A_3 \in \mathbb C$, let $\underline \Pi A_i$ denote the geometric mean of their absolute values; that is,   
\[ 
\underline\prod A_i := \Big( \prod_{i=1}^{3} |A_i| \Big)^{1/3}.
\]
From simple calculations it is easy to see the followings.
If $A$, $A_i$ and $B_i$ are complex numbers for $i=1,2,3$, then 
\begin{align*}
\underline\prod A &= |A|, \\
\underline\prod CA_i &= C \underline\prod A_i \qquad \text{for $C \ge 0$},\\
\underline\prod (A_i B_i) &= \underline\prod A_i \underline\prod  B_i, \\
\underline\prod A_i^{\alpha} &= \Big( \underline\prod A_i \Big)^{\alpha} \qquad \text{for $\alpha \in \mathbb R$}.
\end{align*}
Also, if all $A_{i,\Delta} \in \mathbb C$ and $f_i \in L^p$, then by H\"older's inequality it follows that for $1 \le p \le \infty$,
\begin{align}
\label{Holder1}
\Big( \sum_{\Delta} \underline\prod A_{i,\Delta}^{p} \Big)^{1/p}  &\le \underline\prod \Big( \sum_{\Delta} |A_{i,\Delta}|^p \Big)^{1/p}, \\  
\label{Holder2}
\Big\| \underline\prod f_i \Big\|_p &\le  \underline\prod \|f_i\|_p.
\end{align}

Now we state our trilinear interpolation lemma. 

\begin{lem} \label{lem:MulInterpolation}
Let $2 \le p_1,p_2,q_1,q_2 \le \infty$.
Assume that
\begin{equation} \label{givenEst}
\Big\| \underline\prod f_i \Big\|_{q_1} \le A_1 \underline\prod \| f_i \|_{p_1,\delta},\qquad
\Big\|\underline\prod  f_i \Big\|_{q_2} \le A_2 \underline\prod \| f_i \|_{p_2,\delta}
\end{equation}
for all $f_i$, $i=1,2,3,$ with $\hat f_i \subset \Gamma_\delta$.
Suppose that for some $\theta \in (0,1)$,
\[ 
\frac{1}{q} = \frac{1-\theta}{q_1} + \frac{\theta}{q_2}, \qquad  
\frac{1}{p} = \frac{1-\theta}{p_1} + \frac{\theta}{p_2}
\] and $2 \le p \le q \le \infty$.
Then 
\[ 
\Big\| \underline\prod f_i \Big\|_{q} \lesssim \delta^{-\varepsilon}  A_1^{1-\theta} A_2^{\theta} \underline\prod \| f_i \|_{p,\delta}
\]
for all $f_i$, $i=1,2,3,$ with $\hat f_i \subset \Gamma_\delta$ and all $\varepsilon >0$.
\end{lem}

\begin{proof}
The proof is similar to Lemma \ref{lem:interp}.
%
%
%
We decompose $\underline \prod f_i = \sum_{k \in \delta^{-1} \mathbb  Z^3} \psi_k \underline \prod f_i$ where $\psi_k := \psi(\delta(x-k))$. We can reduce it in an analogous manner to the proof of Lemma \ref{lem:interp}. By localization, it suffices to show that 
\begin{equation} \label{mulinterpol}
\Big\| \underline\prod g_i \Big\|_q  \lesssim \delta^{-\varepsilon} A_1^{1-\theta} A_2^{\theta}  
\end{equation}
for all $g_i := \psi_0 f_i$ with  $\|g_i\|_{p,2\delta} = 1$.
Some minor portions can be removed as in the proof of Lemma \ref{lem:interp}. Since $\psi_0$ decays rapidly outside $B(0,C\delta^{-1})$, we have
\(
\| \underline\prod g_i\|_q \le \| \underline\prod g_i\|_{L^q(B(0,\delta^{-1-\varepsilon}))} + C_K\delta^{K}
\)
for all $\varepsilon>0$ and $K>0$. 
Since $g_i$ is Fourier supported in  $\Gamma_{2\delta}$,
by Lemma \ref{lem:wavepack},
\[
g_i(x) = \sum_{\Theta_i \in \mathbf \Pi_{\delta}} \sum_{\pi_i \in \mathbf P_{\Theta_i}} h_{\pi_i} g_{\pi_i}(x).
\]
By \eqref{ess_supp}, we can eliminate $\pi_i$ that is disjoint from $B(0, C\delta^{-1-\varepsilon})$, so we can restrict $\mathbf P_i$ to the collection $\mathring{\mathbf P}_i$ of  $\pi_i$ intersecting $B(0, C\delta^{-1-\varepsilon})$.  We can also remove $\pi_i$ with $0< h_{\pi_i} \lesssim \delta^{500}$. 
%
%

For dyadic $\delta^{500} \lesssim h_i \lesssim 1$, we define  
\(
\mathring{\mathbf P}_{h_i}:= \{ \pi \in \mathring{\mathbf P}_i: h_i \le h_\pi < 2h_i  \}.
\)
Let
\(
\mathring{\mathbf P}_{\Theta_i}(h_i) := \mathring{\mathbf P}_{h_i} \cap \mathbf P_{\Theta_i}
\), and for any dyadic number $1 \le k_i \lesssim \delta^{-2}$ we define 
\[ 
\mathring{\mathbf P}_{i}(h_i,k_i) = \bigcup_{k_i \le \# \mathring{\mathbf P}_{\Theta_i}(h_i) < 2k_i} \mathring{\mathbf P}_{\Theta_i}(h_i).
\]
Then, we have
\[
\Big\| \underline\prod g_i \Big\|_q \lesssim  \Big\| \underline\prod \Big( \sum_{\delta^{500} \lesssim h_i \lesssim 1} h_i \sum_{1 \lesssim k_i \lesssim \delta^{-2}} \sum_{\pi \in \mathring{\mathbf P}_i(h_i,k_i)}  g_{\pi_i} \Big) \Big\|_q + \delta^{100}. 
\]
We write as
\[
\prod_{i=1}^{3} \Big( \sum_{h_i} h_i \sum_{k_i} \sum_{\pi \in \mathring{\mathbf P}_i(h_i,k_i)}  g_{\pi_i} \Big) 
= \sum_{h_1,h_2,h_3} \sum_{k_1,k_2,k_3} \prod_{i=1}^{3} \Big( h_i \sum_{\pi \in \mathring{\mathbf P}_i(h_i,k_i)}  g_{\pi_i} \Big).
\]
%
By dyadic pigeonholing, there exist dyadic numbers $h_i$ and $k_i$, $i=1,2,3,$ so that
\[ 
\Big\|
\underline \prod \Big( \sum_{\delta^{500} \lesssim h_i \lesssim 1} h_i \sum_{1 \lesssim k_i \lesssim \delta^{-2}} \sum_{\pi \in \mathring{\mathbf P}_i(h_i,k_i)}  g_{\pi_i} \Big)  \Big\|_q \lesssim (\log \delta^{-1})^2 \Big( \underline\prod h_i \Big) \Big\|\underline\prod \Big( \sum_{\pi \in \mathring{\mathbf P}_i(h_i,k_i)}  g_{\pi_i} \Big) \Big\|_q .
\]
Let $\tilde g_i := \sum_{\pi \in \mathring{\mathbf P}_i(h_i,k_i)}  g_{\pi_i} $. Then from these estimates we have
\[
\Big\| \underline\prod g_i \Big\|_q \lesssim \delta^{-\varepsilon} \Big( \underline\prod h_i \Big) \Big\| \underline\prod \tilde g_i \Big\|_q +\delta^{100}.
\]
Since $\tilde g_i$ are balanced functions, from H\"oler's inequality, \eqref{givenEst} and Lemma \ref{lem:revHol} it follows that
\[ 
\Big\|\underline\prod \tilde g_i \Big\|_q \le \Big\|\underline\prod \tilde g_i \Big\|_{q_1}^{1-\theta} \Big\|\underline\prod \tilde g_i \Big\|_{q_2}^{\theta}
\le A_1^{1-\theta} A_2^{\theta}\underline\prod \|\tilde g_i\|_{p_1,2\delta}^{1-\theta} \underline\prod \|g_i\|_{p_2,2\delta}^{\theta} \lesssim A_1^{1-\theta} A_2^{\theta} \delta^{-\varepsilon} \underline\prod \| \tilde g_i\|_{p,2\delta}
\]
and by \eqref{lp_sum},
\[
h_i \| \tilde g_i \|_{p,2\delta} \lesssim \| g_i \|_{p,2\delta}.
\]
Therefore, these estimates yield \eqref{mulinterpol}.
\end{proof}
\begin{remark}
By using analogous methods to Remark \ref{rem:Alt_inter}, we can obtain Lemma \ref{lem:MulInterpolation} without $\epsilon$-losses by known multilinear interpolation theorems, see, e.g., \cite{bergh1976interpolation}. 
\end{remark}

\section{Proof of Theorem \ref{thm:sqfEst}.}

This section is devoted to the proof of  $\mathcal{SQ}(4 \to 4; 1/16)$. 
By Proposition \ref{prop:MSQmeanSQ} this follows from the trilinear square function estimate $\mathcal{SQ}(4 \times 4 \times 4 \to 4;1/16)$. To prove this we will utilize the following two theorems. The first one is the multilinear restriction theorem due to Bennet, Carbery and Tao  \cite{bennett2006multilinear}.

\begin{thm}[Bennet--Carbery--Tao \cite{bennett2006multilinear}] \label{thm:MRT}
Let $f_i$, $i=1,2,3,$ be supported in $\Gamma^{\Omega_i}$.
Suppose that $\Gamma^{\Omega_1}, \Gamma^{\Omega_2}, \Gamma^{\Omega_3}$ are $\nu$-transverse.
 If $R \gg \nu^{-1}$ then for any $\epsilon >0$ and any ball $Q_R$ of radius $R$,
\begin{equation} \label{mrt}
\Big\| \underline\prod \widehat{f_jd\sigma_j} \Big\|_{L^3(Q_R)} \le C_\epsilon R^{\epsilon} \underline\prod \| f_j \|_{2},
\end{equation}
where $d\sigma_j$ is the induced Lebesgue measure on $\Gamma^{\Omega_j}$.
\end{thm}
Note that if the restriction operator $\mathfrak R$ is defined as  the restriction $\mathfrak R f =\hat f \big|_{\Gamma}$ to $\Gamma$ of the Fourier transform $\hat f$, then the \textit{extension  operator} $\widehat{f d\sigma}$ is its adjoint operator $\mathfrak R^*f$. 

\

The second one is the $\ell^2$ decoupling theorem due to Bourgain and Demeter \cite{bourgain2015proof}.

\begin{thm}[Bourgain--Demeter \cite{bourgain2015proof}]\label{thm:Decoupling}
Suppose that the Fourier support of $f$ is contained in $\Gamma_\delta$. Then for any $\epsilon >0$, 
\begin{equation} \label{FrDecp}
\| f \|_{6} \le C_\epsilon \delta^{-\epsilon}\Big(\sum_{\Theta \in \mathbf \Pi_\delta} \| f_\Theta \|_6^2 \Big)^{1/2}.
\end{equation}
\end{thm}

To deal with local estimates we define local norms as follows:
\[
\|f\|_{L^p(\psi_B)} := \|f \psi_B \|_p.
\]
and for any functions $f$ with $\supp \hat f \subset \Gamma_\delta$,
\[ 
\|f\|_{p,\delta,B} := \Big( \sum_{\Theta \in \mathbf \Pi_{\delta}} \| f_{\Theta} \|^2_{L^p(\psi_{B})} \Big)^{1/2}.
\]
Note that if $B$ is a ball of radius $\ge 2/\sqrt\delta$ then for $p\ge 2$,
\begin{equation} \label{locDecNorm}
\| f \psi_{B} \|_{p,\delta} \lesssim \| f \|_{p,\delta,B}.
\end{equation}
Indeed, we decompose the Fourier transform of $(f\psi_B)\ast \Xi_\Theta$ as follows: 
\[
(\hat f \ast \hat \psi_{B}) \hat \Xi_\Theta = ((\hat f \hat \Xi_{C\Theta} ) \ast  \hat \psi_{B}) \hat \Xi_\Theta
+  ((\hat f (1- \hat \Xi_{C\Theta}) ) \ast \hat \psi_{B}) \hat \Xi_\Theta.
\]
Consider the last term of the above equation. We write as
\[
((\hat f (1- \hat \Xi_{C\Theta}) ) \ast \hat \psi_{B})(x) \hat \Xi_\Theta(x) = \int \hat f(y) (1- \hat \Xi_{C\Theta})(y) \hat \psi_{B}(x-y) \hat \Xi_\Theta(x) dy.
\]
For $y \in \Gamma_\delta \setminus C\Theta$ and $x \in \Theta$ we have $|x-y| \ge \sqrt\delta$, and $\hat \psi_{B}$ is supported in a ball of radius $\le \sqrt\delta/2$ with center 0. By considering supports we can see that the above equation is zero. Thus, by Fourier inversion,
\[
(f \psi_{B}) \ast \Xi_\Theta = ((f \ast \Xi_{C\Theta}) \psi_{B}) \ast \Xi_\Theta.
\]
By this equation, Young's inequality and the triangle inequality, we have 
\begin{align*} 
\| (f \psi_{B}) \ast \Xi_\Theta \|_p 
\lesssim  \|(f \ast \Xi_{C\Theta}) \psi_{B} \|_p 
\lesssim \sum_{\Theta' \subset C\Theta}\|(f \ast \Xi_{\Theta'}) \psi_{B} \|_p.
\end{align*}
From this  we can obtain \eqref{locDecNorm}.

\subsection{}
We will deduce a trilinear decoupling estimate from Theorem \ref{thm:MRT} and Theorem \ref{thm:Decoupling}. 
By combining Theorem \ref{thm:MRT} with a localization argument and a slicing argument, it follows that  
\[
\Big\| \underline \prod f_i \Big\|_{3} 
\le C_\epsilon  \delta^{1/2 -\epsilon} \underline\prod \|f_i\|_{2}
\]
for all $f_i$ with $\supp \hat f_i \subset \Gamma_\delta^{\Omega_i}$,
(for the details, see \cite{bennett2006multilinear}, \cite{lee2012cone}, \cite{tao1998bilinear}). By orthogonality, if $f$ is a function with $\supp \hat f \subset \Gamma_\delta$, then
\[
\| f \|_2  \sim \Big( \sum_{\Theta \in \mathbf \Pi_{\delta}} \| f_{\Theta} \|_2^2 \Big)^{1/2} = \| f\|_{2,\delta}.
\]
Thus, we have
\begin{equation*} 
\Big\| \underline \prod f_i \Big\|_{3} 
\le C_\epsilon  \delta^{1/2-\epsilon} \underline\prod \|f_i\|_{2,\delta}.
\end{equation*}
On the other hand, from \eqref{FrDecp} and H\"older's inequality we have
\[
\Big\| \underline \prod f_i \Big\|_{6} 
\le C_\epsilon  \delta^{-\epsilon} \underline\prod \|f_i\|_{6,\delta}.
\]
We interpolate these two estimates by Lemma \ref{lem:MulInterpolation}. Then,
\[ 
\Big\| \underline \prod f_i \Big\|_{4} 
\le C_\epsilon  \delta^{1/4-\epsilon}\underline\prod \|f_i\|_{3,\delta}.
\]
By H\"older's inequality one has $\|f_i\|_{3,\delta} \le \|f_i\|_{4,\delta}^{2/3} \|f_i\|_{2,\delta}^{1/3}$.  Inserting this into the above we obtain
\begin{equation} \label{mainEq}
\Big\| \underline \prod f_i \Big\|_{4} 
\le C_\epsilon  \delta^{1/4-\epsilon} \Big(\underline\prod \|f_i\|_{4,\delta}\Big)^{2/3} \Big( \underline\prod \|f_i\|_{2,\delta} \Big)^{1/3} .
\end{equation}

\subsection{}
Set $R=\delta^{-1}$.
We take a covering $\{\Delta\}$ of $\mathbb R^3$ by finitely overlapping $2R^{1/2}$-balls. We apply the estimate \eqref{mainEq} to $f_i\psi_{\Delta}$. Since  the Fourier support of $f_i \psi_\Delta$ is in $\Gamma_{2\sqrt\delta}$, by  \eqref{mainEq} and \eqref{locDecNorm} we obtain
\[ 
\Big\| \underline \prod f_i \Big\|_{L^4(\Delta)} 
\le C_\epsilon   R^{-1/8 +\epsilon/2} \Big(\underline\prod \|f_i\|_{4,\sqrt\delta,\Delta}\Big)^{2/3} \Big( \underline\prod \|f_i\|_{2,\sqrt\delta,\Delta} \Big)^{1/3}.
\] 
After taking the 4th power in the above, we sum over $\Delta$, and apply H\"older's inequality. Then,
\[ 
\sum_{\Delta} \Big\| \underline \prod f_i \Big\|_{L^4(\Delta)}^4  
\le C_\epsilon  R^{-1/2+2\epsilon} \Big(  \sum_{\Delta} \underline\prod \|f_i\|^4_{4,\sqrt\delta,\Delta} \Big)^{2/3} \Big(  \sum_{\Delta} \underline\prod \|f_i\|^4_{2,\sqrt\delta,\Delta} \Big)^{1/3}.
\]
After taking the 4th root in the above, we apply \eqref{Holder1} to the right-hand sums. Then,
\begin{align*}
\Big( \sum_{\Delta} \Big\| \underline \prod f_i \Big\|_{L^4(\Delta)}^4 \Big)^{1/4} 
&\le C_\epsilon  R^{-1/8+\epsilon/2} \Big( \underline\prod \Big(  \sum_{\Delta}\|f_i\|^4_{4,\sqrt\delta,\Delta} \Big)^{1/4}\Big)^{2/3} \Big(  \underline\prod \Big( \sum_{\Delta}\|f_i\|^4_{2,\sqrt\delta,\Delta} \Big)^{1/4}\Big)^{1/3}.
\end{align*}
We have \(
\big(  \sum_{\Delta}\|f_i\|^4_{4,\sqrt\delta,\Delta} \big)^{1/4} \lesssim  \|f_i\|_{4,\sqrt\delta}
\) by Minkowski's inequality. Thus, from the above estimate it follows that
\begin{equation} \label{eqn:AB}
\Big\| \underline \prod f_i \Big\|_4 
\le C_\epsilon   R^{-1/8 + \epsilon/2} \Big( \underline\prod A_i \Big)^{2/3} \Big( \underline\prod B_i \Big)^{1/3},
\end{equation}
where
\[
A_i := \|f_i\|_{4,\sqrt\delta}, \qquad B_i := \Big( \sum_{\Delta}\|f_i\|^4_{2,\sqrt\delta,\Delta} \Big)^{1/4}.
\]

\subsection{}

We will show that
\begin{equation}\label{L2part}
B_i \lesssim R^{3/8}\| S_\delta f_i\|_{4}.
\end{equation}
By definition we write
\( \|f_i\|_{2,\sqrt\delta,\Delta}^2 = \sum_{\Upsilon \in \mathbf \Pi_{\sqrt\delta}} \|f_{i,\Upsilon}\|_{L^2(\psi_{\Delta})}^2. \)
Since $f_{i,\Upsilon}$ is decomposed as 
\( f_{i,\Upsilon} = \sum_{\Theta \in \mathbf \Pi_{\delta} : \Theta \subset 2 \Upsilon} f_{i,\Theta}, \) we have 
\[
\|f_i\|_{2,\sqrt\delta,\Delta}^2 = \sum_{\Upsilon \in \mathbf \Pi_{\sqrt\delta}}  \int \Big| \sum_{\Theta \in  \mathbf \Pi_{\delta} :\Theta \subset 2\Upsilon} f_{i,\Theta} \psi_\Delta \Big|^2.
\] 
We see that the Fourier support of $f_{i,\Theta} \psi_\Delta$ is contained in the $\delta^{1/2}$-neighborhood of $\Theta$ which is a rectangular box of size $C\delta^{1/2} \times C\delta^{1/2} \times C$ for some constant $C>1$. So, by orthogonality it follows that
\[
\|f_i\|_{2,\sqrt\delta,\Delta}^2 
\lesssim \sum_{\Upsilon \in \mathbf \Pi_{\sqrt\delta}}  \sum_{\Theta \in  \mathbf \Pi_{\delta} :\Theta \subset 2\Upsilon}\int |  f_{i,\Theta} \psi_\Delta |^2  
\lesssim \sum_{\Theta \in \mathbf \Pi_{\delta}} \int |  f_{i,\Theta} \psi_\Delta |^2.
\]
Since \( \sum_{\Theta \in \mathbf \Pi_{\delta}} \int |  f_{i,\Theta} \psi_\Delta |^2 =   \int \big( \sum_{\Theta \in \mathbf \Pi_{\delta}} |  f_{i,\Theta} |^2  \big)^{\frac{1}{2} \times 2} \psi_\Delta^2  = \| S_\delta f_i \|_{L^2(\psi_\Delta)}^2, \)
the above estimate may be written as 
\[
\|f_i\|_{2,\sqrt\delta,\Delta} \lesssim \| S_\delta f_i \|_{L^2(\psi_\Delta)}.
\]
By using this estimate and H\"older's inequality,
\[
B_i \lesssim \Big( \sum_{\Delta}\|S_\delta f_i\|^4_{L^2(\psi_{\Delta})} \Big)^{1/4}
\lesssim R^{\frac{3}{2}\big( \frac{1}{2} - \frac{1}{4} \big)} \Big( \sum_{\Delta}\|S_\delta f_i\|^4_{L^4(\psi_{\Delta})} \Big)^{1/4} \lesssim R^{3/8}\| S_\delta f_i\|_{4}.
\]
Thus we obtain \eqref{L2part}.

\subsection{}
Let $\alpha \ge 0$ be the best constant such that $\mathcal{SQ}(4 \times 4 \times 4 \rightarrow 4;\alpha)$, i.e.,
\[
\alpha = \inf_{\delta > 0} \Big( \log_{1/\delta} \sup_{f_i:\supp \hat f_i \subset \Gamma_\delta^{\Omega_i}} \frac{\| \underline\prod f_i \|_4}  {\underline\prod \|S_\delta f_i \|_4} \Big).
\]
To prove  $\mathcal{SQ}(4 \times 4 \times 4 \rightarrow 4;1/16)$ it is enough to show that for any $\epsilon>0$,
\[
\alpha \le \frac{1}{16} +C\epsilon. 
\]
By H\"older's inequality,
\[
A_i \lesssim R^{\frac{1}{4} \big( \frac{1}{2} - \frac{1}{4} \big)} \Big( \sum_{\Upsilon \in \mathbf \Pi_{\sqrt\delta}} \| f_{i,\Upsilon} \|_{4}^{4} \Big)^{1/4}.
\]
By the definition of $\alpha$ and Proposition \ref{prop:MSQmeanSQ} one has $SQ(4 \to 4; \alpha)$. 
By Lorentz rescaling, as in \eqref{ppsc},
\[ 
\| f_{i,\Upsilon} \|_4 \le C_\epsilon R^{\alpha/2 + \epsilon} \| S_\delta f_{i,\Upsilon} \|_4.
\]
So, we have
\[
A_i
\le C_\epsilon R^{\alpha/2 + \epsilon} R^{\frac{1}{4} \big( \frac{1}{2} - \frac{1}{4} \big)}  \Big( \sum_{\Upsilon \in \mathbf \Pi_{\sqrt\delta}} \| S_\delta f_{i,\Upsilon} \|_4^4 \Big)^{1/4} .
\]
Since 
\begin{align*}
\sum_{\Upsilon \in \mathbf \Pi_{\sqrt\delta}} \| S_\delta f_{i,\Upsilon} \|_4^4 &\lesssim \sum_{\Upsilon \in \mathbf \Pi_{\sqrt\delta}}  \int \Big( \sum_{\Theta \in  \mathbf \Pi_{\delta} :\Theta \subset 2\Upsilon} |f_{i,\Theta}|^2 \Big)^{2} \\
&\lesssim \int \Big( \sum_{\Upsilon \in \mathbf \Pi_{\sqrt\delta}} \sum_{\Theta \in  \mathbf \Pi_{\delta} :\Theta \subset 2\Upsilon} |f_{i,\Theta}|^2 \Big)^{2} \\
&\lesssim  \| S_\delta f_i \|_4^4,
\end{align*}
we obtain
\begin{equation} \label{APart}
A_i \le C_\epsilon  R^{1/16 +\alpha/2+\epsilon} \| Sf_i \|_4.
\end{equation}

Now we insert \eqref{APart} and \eqref{L2part} into \eqref{eqn:AB}. Then,
\[ 
\Big\| \underline \prod f_i \Big\|_{L^4(Q_R)} \le C_\epsilon R^{1/24+\alpha/3 + C\epsilon} \underline\prod \|Sf_i \|_4.
\]
Since  $\alpha$ is the best constant holding $SQ(4 \times 4 \times 4 \rightarrow 4;\alpha)$, we have
\( \alpha \le \frac{1}{24} + \frac{\alpha}{3} +C\epsilon. \)
Therefore, \(  \alpha \le \frac{1}{16} +C\epsilon. \) This completes the proof.

\section{Proof of Theorem \ref{thm:localSm}.}

In this section, Theorem \ref{thm:localSm} will be proved by using a corresponding trilinear estimate. 
Let us define an operator $U_N$ by
\[ 
U_N f(x,t) = \check \eta_N \ast  e^{it \sqrt{-\Delta}} f(x)
\]
where $\eta_N$ is a bump function supported in $\{\xi \in \mathbb R^2 : |\xi| \sim N\}$ and $\check \eta_N$ is the inverse Fourier transform of $\eta_N$. 
By the Littlewood--Paley decomposition, to prove Theorem \ref{thm:localSm} it suffices to show that the estimate
\[
\|U_Nf \|_{L^{30/7}(\mathbb R^2 \times [1,2])} \le C_\epsilon N^{1/10+\epsilon} \|f\|_{10/3}
\]
holds for all $\epsilon>0$, all $N \ge 1$ and all $f \in L^{10/3}(\mathbb R^2)$.

For convenience of rescaling we reform $U_Nf$ as follows.
By a linear transformation $J:(\xi_1,\xi_2,\xi_3) \mapsto (\zeta_1,\zeta_2,\zeta_3) = ({\xi_3-\xi_1},\xi_2,{\xi_3+\xi_1})$ which maps the cone $\{ (\xi_1,\xi_2,\pm \sqrt{\xi_1^2+\xi_2^2}) \}$ to the leaned cone $\{ (\zeta_1, \zeta_2,\zeta^2_2/\zeta_1 \}$, we redefine $U_Nf$ by
\begin{equation} \label{eqn:U_N}
U_N f(x,t) = \int e^{2\pi i (x \cdot \xi +t {\xi_2^2}/{\xi_1})} \hat f(\xi) \eta_{N}(\xi_1) \varphi(\xi_2/\xi_1) d\xi, \qquad \xi=(\xi_1,\xi_2),
\end{equation}
where $\varphi$ is a bump function supported in the unit interval. Then, ${U_Nf}$ has Fourier support in  
\[
\Gamma(N) := \{(\xi_1,\xi_2,\xi_2^2/\xi_1) : |\xi_1|\sim N,~ |\xi_2/\xi_1| \lesssim 1\}.
\] 
The leaned cone $(\xi_1,\xi_2,\xi_2^2/\xi_1)$ is written as $\xi_1(1,\theta,\theta^2)$ where $\theta = \xi_2 /\xi_1$. So one may identify $\theta$ with an angular variable of the cone. 

We say that the local smoothing estimate $\mathcal{LS}(p \to q; \alpha)$ holds if 
\begin{equation} \label{unf}
\| U_Nf \|_{L^q(\mathbb R^2 \times [1,2])} \le C_{\epsilon} N^{\alpha+\epsilon} \|f\|_{p}
\end{equation}
holds for all $\epsilon>0$, all $N > 1$ and all $f \in L^p(\mathbb R^2)$. To prove Theorem \ref{thm:localSm} it suffices to show 
\[
\mathcal{LS}(10/3 \to 30/7; 1/10).
\]
For given $1 \le p < q \le \infty$ and $\frac{1}{p} + \frac{3}{q} =1$, we define 
\begin{equation} \label{alpCon}
\alpha= \alpha(p,q) \ge \frac{1}{p} - \frac{3}{q} +\frac{1}{2}
\end{equation}
to be the best exponent for which the estimate \eqref{unf} holds for all $N > 1$ and all $f \in L^p(\mathbb R^2)$, i.e.,
\[
\alpha(p,q) = \inf_{N > 1} \Big( \log_N  \sup_{f \in L^p(\mathbb R^2)} \frac{\|U_N f\|_{L^q(\mathbb R^2 \times [1,2])}}{\|f\|_p} \Big).
\]
Then it is enough to show that for all $\epsilon,~ \epsilon_1>0$,
\begin{equation} \label{alG}
\alpha\Big(\frac{10}{3}, \frac{30}{7} \Big) \le \frac{1}{10} + C\epsilon_1 + \log_N C_{\epsilon, \epsilon_1},
\end{equation}
since we may take $\epsilon = \epsilon_1$, which can be absorbed in an $\epsilon$-loss in \eqref{unf}.

\

\subsection{}
Let an arbitrary small $\epsilon_1>0$ be given. 
Let $N \ge N_0$ and  \( 1 > \gamma_1  > \gamma_2 \ge N_0^{-\epsilon_1/2} \). Later,  $\gamma_1$, $\gamma_2$ and $N_0$ will be chosen.
By rescaling  and (a minor variant of) Lemma \ref{lem:LinTriLcompare} one has that for any $(x,t) \in \mathbb R^2 \times [1,2]$,
\begin{align*}  
|U_N f(x,t) | &\lesssim  \max_{\Omega \in \mathbf \Omega(\gamma_1)} | U_N^{\Omega} f(x,t)|
+  \gamma_1^{-1} 
\max_{\Omega \in \mathbf \Omega(\gamma_2)}	| U_N^{\Omega} f(x,t) |\\
&\qquad \qquad + \gamma_2^{-50} 
\max_{\substack{\Omega_1,\Omega_2,\Omega_3 \in \mathbf 
\Omega(\gamma_2): \\ \dist(\Omega_i, \Omega_j) \ge \gamma_2,\, i 
\neq j}} \Big| \Big( \prod_{i=1}^{3} |U_N^{\Omega_i} f(x,t)| \Big)^{1/3} \Big|,
\end{align*}
where $U_N^{\Omega}$ is defined as \eqref{eqn:U_N} with $\varphi$ replaced by $\varphi_\Omega$ which is a bump function supported in $\Omega$.

By embedding $\ell^{q} \subset \ell^{\infty}$ it follows that
\begin{equation} \label{LpTricom1} 
\begin{split}  
\|U_N f \|_{L^q(\mathbb R^2 \times I)} & \lesssim \Big( \sum_{\Omega_1 \in \mathbf \Omega(\gamma_1)} \| U_N^{\Omega_1} f \|^q_{L^q(\mathbb R^2 \times I)} \Big)^{1/q}
+  \gamma_1^{-1} \Big(
\sum_{\Omega_2 \in \mathbf \Omega(\gamma_2)}	\| U_N^{\Omega_2} f \|^q_{L^q(\mathbb R^2 \times I)} \Big)^{1/q} \\
&\qquad \qquad + \gamma_2^{-50} \Big(
\sum_{\substack{\Omega_1,\Omega_2,\Omega_3 \in \mathbf 
\Omega(\gamma_2): \\ \dist(\Omega_i, \Omega_j) \ge \gamma_2,\, i 
\neq j}} \Big\| \Big( \prod_{i=1}^{3} |U_N^{\Omega_i} f_i| \Big)^{1/3} \Big\|^q_{L^q(\mathbb R^2 \times I)} \Big)^{1/q},
\end{split}
\end{equation}
where $I=[1,2]$.

We consider the first and second summation in the right-hand side of \eqref{LpTricom1}.
From rescaling and the definition of $\alpha$ it follows that
\begin{equation} \label{LShyp}
\| U_N^{\Omega_i} f\|_{L^q(\mathbb R^2 \times I)} \le C\gamma_i^{3\big(\frac{1}{q}-\frac{1}{p} \big)}  (\gamma_i^2 N)^{\alpha+\epsilon} \| f \|_p.
\end{equation}
More specifically, by rotating we may assume that $\Omega$ is centered at 0. Then we may write $U_N^{\Omega_i} f$ as 
\[
U_N^{\Omega_i} f(x,t) = \int e^{2\pi i (x \cdot \xi +t {\xi_2^2}/{\xi_1})} \hat f(\xi) \eta_N(\xi_1) \varphi(\gamma_i^{-1} \xi_2/\xi_1)  d\xi.
\]
Let $\sigma(x_1,x_2,t) = (\gamma_i^2 x_1, \gamma_i x_2, t)$ and $\underline \sigma(x_1,x_2) = (\gamma_i^2 x_1, \gamma_i x_2)$.  
Then, we have
\(
U_N^{\Omega_i} f \circ \sigma = U_{\gamma_i^2 N}(f \circ \underline\sigma).
\)
Thus, using \eqref{unf} and this relation we have \eqref{LShyp}.

If we define $f_\Omega$ by 
\begin{equation*} 
\widehat f_\Omega(\xi_1,\xi_2) = \hat f(\xi_1,\xi_2)  \chi_{\{|\xi_1| \sim N\}}(\xi_1) \chi_{\Omega}(\xi_2/\xi_1),
\end{equation*}
then we may replace $U_N^{\Omega_i} f$ with $U_N^{\Omega_i} f_{\Omega_i}$, where $\chi$ denotes a characteristic function. By \eqref{LShyp},
\begin{equation} \label{eqn:sum_of_U_N}
\Big( \sum_{\Omega_i \in \mathbf \Omega(\gamma_i)} \| U_N^{\Omega_i} f_{\Omega_i} \|_q^q \Big)^{1/q} 
\le C \gamma_i^{3\big(\frac{1}{q} - \frac{1}{p} \big)} (\gamma_i^{2}N)^{\alpha+\epsilon} \Big(\sum_{\Omega_i \in \mathbf \Omega(\gamma_i)} \| f_{\Omega_i} \|_p^q \Big)^{1/q}.
\end{equation}

We recall the following lemma from \cite{tao2000bilinearII}.

\begin{lem}[\cite{tao2000bilinearII}*{Lemma 7.1}] \label{lem:desum}
Let $R_k$ be a collection of rectangles such that the dilates $2R_k$ are almost disjoint, and suppose that $f_k$ are a collection of functions whose Fourier transforms are supported on $R_k$. Then for all $1 \le p \le \infty$ we have
\[
\Big( \sum_{k} \|f_k\|_p^{p^*} \Big)^{1/p^*} \lesssim \Big\| \sum_{k} f_k \Big\|_p \lesssim \Big( \sum_k \|f_k \|_p^{p_*} \Big)^{1/p_*},
\]
where $p_* = \min(p,p')$, $p^* = \max(p,p')$.
\end{lem}
It is remarked that Lemma \ref{lem:desum} is elementary, and simply a consequence of interpolation between Plancherel's theorem and Minkowski's inequality for the $L^\infty$ space.

After embedding $\ell^p \subset \ell^q$ in the right-hand side of \eqref{eqn:sum_of_U_N}, we apply Lemma \ref{lem:desum}. Then we obtain
\begin{equation} \label{scEst}
\Big( \sum_{\Omega_i \in \mathbf \Omega(\gamma_i)} \| U_N^{\Omega_i} f \|_{L^q(\mathbb R^2 \times I)}^q \Big)^{1/q} \le C \gamma_i^{3\big(\frac{1}{q} - \frac{1}{p} \big)} (\gamma_i^2 N)^{\alpha+\epsilon} \|f\|_p.
\end{equation}

\

\subsection{}
We consider the last summation in  the right-hand side of \eqref{LpTricom1}.
We will show that for any $\epsilon>0$,
\begin{equation} \label{GTriLS}
\Big\| \underline \prod U_N^{\Omega_i} f \Big\|_{L^{30/7}(\mathbb R^2 \times I)} \le C_\epsilon  N^{1/10+\epsilon} \|f\|_{10/3}.
\end{equation}
First we prove a corresponding local estimate.

\begin{lem}
Let $B$ be a unit ball. Then, for any $\epsilon > 0$,
\begin{equation} \label{localTLS}
\Big\| \underline\prod |U_N^{\Omega_i} f_i| \Big\|_{L^{30/7}(B \times I)} \le C_\epsilon N^{1/10+\epsilon} \underline\prod \|f_i\|_{10/3}.
\end{equation}
\end{lem}

\begin{proof}
By interpolation it suffices to show
\begin{align} \label{LSdec}
\Big\| \underline\prod U_N^{\Omega_i} f_i \Big\|_{L^6(B \times I)} &\le C_\epsilon N^{1/6+\epsilon} \underline\prod \|f_i\|_{6},\\
\label{LSMR}
\Big\| \underline\prod U_N^{\Omega_i} f_i \Big\|_{L^3(B \times I)} &\le C_\epsilon N^{\epsilon} \underline\prod \|f_i\|_{2}.
\end{align}
Consider \eqref{LSdec}. By H\"older's inequality it is enough to show 
\begin{equation} \label{LSdec2}
\| U_Nf \|_{L^6(B \times I)} 
\le C_\epsilon N^{1/6+\epsilon} \|f\|_6.
\end{equation}
Since $\psi_{I}(t)U_Nf(x,t) $ has Fourier support in a $C$-neighborhood of $\Gamma(N)$, from Theorem \ref{thm:Decoupling} and rescaling it follows that
\[
\| U_Nf \|_{L^6(B \times I)} 
\le C_\epsilon N^{\epsilon} \Big( \sum_{\widetilde\Theta} \|(\psi_I U_Nf) \ast \Xi_{\widetilde\Theta}\|_{6}^2 \Big)^{1/2},
\]
where 
$\widetilde\Theta$ is a sector of size $CN^{1/2} \times CN \times C$.
By H\"older's inequality, this is bounded by
\[
\le C_\epsilon  N^{1/6+\epsilon} \Big( \sum_{\widetilde\Theta} \|(\psi_I U_Nf) \ast \Xi_{\widetilde\Theta}\|_{6}^6 \Big)^{1/6}.
\]
It is well known (see, e.g., \cite{wolff2000local}*{Lemma 6.1}, \cite{stein1993harmonic}*{XI: 4.13}, \cite{mockenhaupt1992wave}) that for $p \ge 2$, 
\[
\Big( \sum_{\widetilde\Theta} \|(\psi_I U_Nf )\ast \Xi_{\widetilde\Theta}\|_{p}^p \Big)^{1/p} \lesssim \|f\|_p.
\]
Thus, we obtain \eqref{LSdec2}

Consider \eqref{LSMR}. In \eqref{mrt}, the restriction operator $\widehat{f_j d\sigma_j}$ can be replaced with $U_1^{\Omega_j} \check f$ where $\check f$ denotes the inverse Fourier transform of $f$. Thus, from Theorem \ref{thm:MRT} and Plancherel's theorem it follows that
\[
\Big\| \underline\prod U_1^{\Omega_i} f_i \Big\|_{L^3(Q_N)} \le C_\epsilon N^{\epsilon} \underline\prod \|f_i\|_{2}.
\]
If $s(x,t) = N^{-1}(x,t)$ and $\underline s (x) = N^{-1}x$, then $U_N^{\Omega} f \circ s = U_1^{\Omega} (f \circ \underline s)$. So, by changing variables and translation invariance, the above estimate gives \eqref{LSMR}.
\end{proof}

We now prove that \eqref{localTLS} implies \eqref{GTriLS}. This immediately follows from the next localization lemma. 

\begin{lem}
Suppose that the local estimate 
\begin{equation} \label{ifloc}
\Big\| \underline\prod U_N^{\Omega_i} f_i  \Big\|_{L^q(B \times I)} \le A(N) \underline\prod \| f_i  \|_{p}
\end{equation}
holds for all unit cubes $B$ and all $f_i \in L^p(\mathbb R^2)$.
If $p \le q$ then the estimate
\begin{equation} \label{thenGl}
\Big\| \underline\prod U_N^{\Omega_i} f_i \Big\|_{L^q(\mathbb R^2 \times I)} \le C N^{\epsilon} A(N) \underline\prod \| f_i  \|_{p} 
\end{equation}
holds for all $\epsilon>0$ and all $f_i \in L^p(\mathbb R^2)$.
\end{lem}

\begin{proof}
We write as 
\[ 
U_N f(x,t) 
= (K_N(t) \ast f)(x)
\]
where
\[ 
K_N(t)(x) = K_N(x,t) 
:= \int e^{2\pi i (x \cdot \xi +t {\xi_2^2}/{\xi_1})} \eta_N(\xi_1) \varphi(\xi_2/\xi_1) d\xi.
\]
By using a stationary phase method, it follows that for $(x,t) \in \mathbb R^2 \times I$,
\[
|K_N(t)(x)| \le C_M N^{2} (1+ |x|)^{-M} \qquad \forall M>0.
\]
Thus, for $(x,t) \in \mathbb R^2 \times I$,
\begin{equation} \label{asyt}
|U_N f(x,t)| 
\le  C_M  ( a_N \ast |f| )(x),  \quad \forall M>0,
\end{equation}
where $a_N(x) = N^2 (1+ |x|)^{-M}$.
%
%
%

If a unit lattice square $B \subset \mathbb R^2$ is given, then we decompose 
\begin{equation} \label{ptLoc}
|U_N f|\chi_{B \times I} \lesssim  |U_N(f \chi_{N^{\epsilon}B})| \chi_{B \times I} + C_M |\mathcal E_{B^c}f| \chi_{B \times I},
\end{equation}
where
\[
\mathcal E_{B^c} f := a_N \ast (|f|\chi_{\mathbb R^2 \setminus N^{\epsilon}B}).
\]
Consider $|\mathcal E_{B^c}f| \chi_{B \times I}$. If $|x-y| \gtrsim N^{\epsilon}$ then one has $a_N(x-y) \lesssim N^2N^{-\epsilon M} \le N^{-2000C}$. 
So, we have
\begin{align*}
\chi_B(x) \big( a_N \ast (|f|\chi_{\mathbb R^2 \setminus N^{\epsilon}B} ) \big)(x) &= \chi_B(x) \int a_N(x-y)\chi_{\mathbb R^2 \setminus N^{\epsilon}B}(y) |f(y)| dy \\
&\lesssim N^{-1000C} \chi_B(x) \int a_N^{1/2}(x-y) |f(y)| dy \\
&\lesssim N^{-1000C}  \chi_B(x) (a^{1/2}_N \ast |f|)(x).
\end{align*}
Thus, by Young's inequality we obtain
\begin{equation} \label{errest}
\Big( \sum_{B} \| \mathcal E_{B^c} f\|_{L^q(B)}^q \Big)^{1/q}
\lesssim N^{-900C}   \|f \big\|_{p}.
\end{equation}

On the other hand, by some rough estimates (cf. Young's inequality) we see that
$\|U_N f \|_{L^q({B \times I})} \lesssim N^C \|f\|_p$. So, by embedding $\ell^p \subset \ell^q$, we have
\begin{equation} \label{roughEst}
\Big( \sum_{B} \|  U_N(f \chi_{N^\epsilon B})  \|_{L^q({B \times I})}^q \Big)^{1/q}
\lesssim N^{C} \Big( \sum_{B} \|f \big\|_{L^p(N^\epsilon B)}^q \Big)^{1/q}
\lesssim N^{2C} \|f\|_p.
\end{equation}

Now, we consider the estimate \eqref{thenGl} by using \eqref{errest} and \eqref{roughEst} . 
We define $f_{\Omega_i}$ as  
\[
\widehat f_{\Omega_i}(\xi_1,\xi_2) = \hat f_i(\xi) \eta_N(\xi_1) \varphi_{\Omega_{i}}(\xi_2/\xi_1).
\]
Then we may replace $U_N^{\Omega_i} f_i$ with $U_N f_{\Omega_i}$. By \eqref{ptLoc}, 
\begin{align}
\underline\prod U_N f_{\Omega_i} \chi_{B \times I} &\lesssim 
\underline\prod  \Big(  |U_N(f_{\Omega_i} \chi_{N^{\epsilon}B})| \chi_{B \times I} + C_M (\mathcal E_{B^c}f_{\Omega_i} )\chi_{B \times I}\Big) \nonumber\\
\label{dle}
&\lesssim 
\underline\prod  |U_N(f_{\Omega_i} \chi_{N^{\epsilon}B})| \chi_{B \times I} + C_M \mathcal E(f_{\Omega_1},f_{\Omega_2},f_{\Omega_3})\chi_{B \times I},
\end{align}
where
\begin{align*}
\mathcal E(f_{\Omega_1},f_{\Omega_2},f_{\Omega_3}) &:= \sum_{i,j,k \in \{1,2,3\}} \big(\mathcal E_{B^c}f_{\Omega_i}  |U_N(f_{\Omega_j} \chi_{N^{\epsilon}B})||U_N(f_{\Omega_k} \chi_{N^{\epsilon}B})| \big)^{1/3} \\
&+ \sum_{i,j,k \in \{1,2,3\}} \big( \mathcal E_{B^c}f_{\Omega_i} \mathcal E_{B^c}f_{\Omega_j}  |U_N(f_{\Omega_k} \chi_{N^{\epsilon}B})| \big)^{1/3}
+ \underline\prod \mathcal E_{B^c}f_{\Omega_i}.
\end{align*}
By Minkowski's inequality,
\begin{equation} \label{errsum}
\begin{split}
&\Big( \sum_{B} \| \mathcal E(f_{\Omega_1},f_{\Omega_2},f_{\Omega_3}) \|_{L^q({B \times I})}^q \Big)^{1/q} \\
&\qquad \lesssim \max_{i,j,k}\Big( \sum_{B} \| \big(\mathcal E_{B^c}f_{\Omega_i}  |U_N(f_{\Omega_j} \chi_{N^{\epsilon}B})||U_N(f_{\Omega_k} \chi_{N^{\epsilon}B})| \big)^{1/3} \|_{L^q({B \times I})}^q \Big)^{1/q} \\
&\qquad\qquad + \max_{i,j,k}\Big( \sum_{B} \| \big( \mathcal E_{B^c}f_{\Omega_i} \mathcal E_{B^c}f_{\Omega_j}  |U_N(f_{\Omega_k} \chi_{N^{\epsilon}B})| \big)^{1/3} \|_{L^q({B \times I})}^q \Big)^{1/q} \\
&\qquad\qquad\qquad +\Big( \sum_{B} \Big\|  \underline\prod \mathcal E_{B^c}f_{\Omega_i} \Big\|_{L^q({B \times I})}^q \Big)^{1/q}.
\end{split}
\end{equation}
Consider the right-hand side of \eqref{errsum}.
By H\"older's inequality,
\begin{multline*}
\Big( \sum_{B} \| \big(\mathcal E_{B^c}f_{\Omega_i}  |U_N(f_{\Omega_j} \chi_{N^{\epsilon}B})||U_N(f_{\Omega_k} \chi_{N^{\epsilon}B})| \big)^{1/3} \|_{L^q({B \times I})}^q \Big)^{1/q} \\
\le \Big( \sum_{B} \| \mathcal E_{B^c}f_{\Omega_i} \|_{L^q({B \times I})}^q \Big)^{1/3q} \Big( \sum_{B} \| U_N(f_{\Omega_j} \chi_{N^{\epsilon}B})   \|_{L^q({B \times I})}^q \Big)^{1/3q} \\
\times \Big( \sum_{B} \| U_N(f_{\Omega_k} \chi_{N^{\epsilon}B})   \|_{L^q({B \times I})}^q \Big)^{1/3q}.
\end{multline*}
Thus, by \eqref{errest} and \eqref{roughEst} it is bounded by 
\[
\lesssim N^{-200C} \underline\prod \| f_{\Omega_i} \|_p.
\]
The second and third summations in the right-hand side of \eqref{errsum} are estimated by an analogous method. Thus, 
\begin{equation} \label{mulerEst}
\Big( \sum_{B} \| \mathcal E(f_{\Omega_1},f_{\Omega_2},f_{\Omega_3}) \|_{L^q({B \times I})}^q \Big)^{1/q}  \lesssim N^{-200C} \underline\prod \| f_{\Omega_i} \|_p.
\end{equation}
By \eqref{dle}
\begin{align*}
\Big\| \underline\prod U_N f_{\Omega_i} \Big\|_{L^q(\mathbb R^2 \times I)} &= \Big( \sum_{B} \Big\| \underline\prod |U_N f_{\Omega_i}  \Big\|_{L^q({B \times I})}^q \Big)^{1/q} \\
&\lesssim \Big( \sum_{B} \Big\| \underline\prod  U_N(f_{\Omega_i} \chi_{N^{\epsilon}B})  \Big\|_{L^q({B \times I})}^q \Big)^{1/q}  
+ \Big( \sum_{B} \| \mathcal E(f_{\Omega_1},f_{\Omega_2},f_{\Omega_3}) \|_{L^q({B \times I})}^q \Big)^{1/q}.
\end{align*}
By \eqref{ifloc}, \eqref{mulerEst} and embedding  $\ell^p \subset \ell^q$, it follows that
\[
\Big\| \underline\prod U_N f_{\Omega_i} \Big\|_{L^q(\mathbb R^2 \times I)}  \lesssim (N^{\epsilon}A(N)+ N^{-200C}) \underline\prod \| f_{\Omega_i} \|_p.
\]
Since $\|f_{\Omega_i}\|_p \lesssim \|f_i\|_p$ by Young's inequality, we obtain \eqref{thenGl}.
\end{proof}

\

\subsection{}
Last of all, we will show \eqref{alG}. 
We substitute \eqref{scEst} and \eqref{GTriLS} in \eqref{LpTricom1} with $(p,q)=(10/3,30/7)$.  Then, it follows that
\begin{equation}
\|U_N f \|_{L^{30/7}(I \times \mathbb R^2)} \lesssim (\gamma_1^{2\alpha - \frac{1}{5}+2\epsilon} N^{\alpha+\epsilon} + \gamma_1^{-1} \gamma_2^{2\alpha - \frac{1}{5} +2\epsilon} N^{\alpha+\epsilon} + C_{\epsilon, \epsilon_1} \gamma_2^{-60} N^{\frac{1}{10}+\epsilon}) \|f\|_{10/3}.
\end{equation}
So, by the assumption that $\alpha$ is a best exponent,
\begin{align*} 
N^{\alpha} \le C(\gamma_1^{2\alpha - \frac{1}{5}+2\epsilon} + \gamma_1^{-1} \gamma_2^{2\alpha - \frac{1}{5}+2\epsilon} ) N^{\alpha} + C_{\epsilon, \epsilon_1}\gamma_2^{-60} N^{\frac{1}{10}}. 
\end{align*}
Observe that $2\alpha -\frac{1}{5} \ge 0$ by \eqref{alpCon}.
We now choose $\gamma_1$, $\gamma_2$ and $N_0$ so that  $C\gamma_1^{2\alpha - \frac{1}{5}+2\epsilon} \le 1/4$, $ C\gamma_1^{-1} \gamma_{2}^{2\alpha - \frac{1}{5}+2\epsilon}  \le 1/4$ and $1> \gamma_1 > \gamma_2 \ge N_0^{-{\epsilon_1}/{2}}$.
Then $N^{\alpha} \le C_{\epsilon,\epsilon_1} N^{\frac{1}{10} + 30\epsilon_1} $. Thus we obtain \eqref{alG}.

\section{Acknowledgments}
The author is indebted to the anonymous referee
whose comments helped improve the presentation of the work. 
The author would like to thank Andreas Seeger for informing his work with Malabika Pramanik.

\
 

\end{document}